\documentclass[11pt]{amsart}
\usepackage{hyperref}
\usepackage[capitalize, nameinlink]{cleveref}
\crefname{theorem}{Theorem}{Theorems}
\crefname{proposition}{Proposition}{Propositions}
\crefname{observation}{Observation}{Observations}
\crefname{lemma}{Lemma}{Lemmas}
\crefname{claim}{Claim}{Claims}
\crefname{problem}{Problem}{Problems}
\crefname{conjecture}{Conjecture}{Conjectures}
\crefname{question}{Question}{Questions}
\crefname{example}{Example}{Examples}
\crefname{fact}{Fact}{Facts}

\usepackage{amssymb}
\usepackage[bbgreekl]{mathbbol}
\usepackage{graphicx}
\usepackage{ifthen}
\usepackage{pict2e}
\usepackage{xargs}
\usepackage{xspace}
\usepackage{xcolor}
\usepackage{pgf,tikz}
\usepackage{pgfplots}
\usepackage{environ}
\usepackage{caption}
\usetikzlibrary{arrows}
\usepackage{enumitem}
\usepackage{xifthen}
\usepackage{comment}
\newcounter{dummy}
\makeatletter
\newcommand\myitem[1][]{\item[#1]\refstepcounter{dummy}\def\@currentlabel{#1}}
\makeatother

\makeatletter
\newsavebox{\measure@tikzpicture}
\NewEnviron{scaletikzpicturetowidth}[1]{%
	\def\tikz@width{#1}%
	\begin{lrbox}{\measure@tikzpicture}%
		\BODY
	\end{lrbox}%
	\pgfmathparse{#1/\wd\measure@tikzpicture}%
	\BODY
}
\makeatother

\DeclareSymbolFontAlphabet{\mathbb}{AMSb}

\newcommand{\thistheoremname}{}
\newtheorem*{genericthm*}{\thistheoremname}
\newenvironment{namedthm*}[1]
{\renewcommand{\thistheoremname}{#1}%
	\begin{genericthm*}}
	{\end{genericthm*}}

\newcommand{\Bairespace}[1][]{
	\ifthenelse{\equal{#1}{}}{\functions{\N}{\N}}{\functions{#1}{\N}}
}
\newcommand{\bbL}{\mathbb{L}}
\newcommand{\bbX}{\mathbb{X}}

\newcommand{\Cantorspace}[1][]{
	\ifthenelse{\equal{#1}{}}{\functions{\N}{2}}{\functions{#1}{2}}
}

\newcommandx{\concatenation}[2][1 = undefined, 2 = undefined]{
	\ifthenelse{\equal{#1}{undefined}}{{}\smallfrown}{
		\ifthenelse{\equal{#2}{undefined}}{\bigoplus #1}{\bigoplus_{#1} #2}
	}
}

\newcommandx{\functions}[3][3 =]{
	\ifthenelse{\equal{#3}{}}{#2^{#1}}{#2_{#3}^{#1}}
}

\newcommand{\Gzero}[1][]{
	\ifthenelse{\equal{#1}{}}
	{\mathbb{G}_0}
	{\mathbb{G}_{0,n}}
}
\newcommandx{\Hzero}[2][2 = undefined]{
	\ifthenelse{\equal{#2}{undefined}}
	{\mathbb{H}_{#1}}
	{\mathbb{H}_{#1, #2}}
}
\newcommandx{\intersection}[2][1 =, 2 =]{
	\ifthenelse{\equal{#1}{}}{\cap}{
		\ifthenelse{\equal{#2}{}}{\bigcap #1}{{\bigcap_{#1} #2}}
	}
}
\newcommand{\Lzero}[1][]{\ifthenelse{\equal{#1}{}}{\bbL_0}{L_{0, #1}}}
\newcommand{\Lzerospace}[1][]{\ifthenelse{\equal{#1}{}}{\bbX_0}{X_{0, #1}}}

\newcommand{\modulo}[1]{\ (\text{mod } 2)}
\newcommand{\N}{\mathbb{N}}

\newcommandx{\product}[2][1 =, 2 =]{
	\ifthenelse{\equal{#1}{}}{\times}{
		\ifthenelse{\equal{#2}{}}{\prod #1}{{\prod_{#1} #2}}
	}
}

\newcommandx{\sequence}[2][2 = undefined]{
	\ifthenelse{\equal{#2}{undefined}}{(#1)}{
		(#1)_{#2}
	}
}

\newcommandx{\set}[2][2 = undefined]{
	\ifthenelse{\equal{#2}{undefined}}{\{ #1 \}}{
		\{ #1 \suchthat #2 \}
	}
}
\newcommandx{\sets}[3][3 =]{
	\ifthenelse{\equal{#3}{}}{[#2]^{#1}}{[#2]^{#1}_{#3}}
}

\newcommand{\suchthat}{\mid}

\renewcommand{\restriction}[2]{#1 \upharpoonright #2}

\newcommandx{\union}[2][1 =, 2 =]{
	\ifthenelse{\equal{#1}{}}{\cup}{
		\ifthenelse{\equal{#2}{}}{\bigcup #1}{{\bigcup_{#1} #2}}
	}
}

\newtheorem{theorem}{Theorem}[section]
\newtheorem{lemma}[theorem]{Lemma}

\newtheorem{claim}[theorem]{Claim}
\newtheorem{subclaim}[theorem]{Subclaim}
\newtheorem{corollary}[theorem]{Corollary}
\newtheorem{proposition}[theorem]{Proposition}
\newtheorem{conjecture}[theorem]{Conjecture}

\newtheorem{problem}[theorem]{Problem}

\theoremstyle{definition}

\newtheorem{definition}[theorem]{Definition}

\numberwithin{equation}{section}

\newcommand{\oom}{\mathbb{N}^\mathbb{N}}

\newcommand{\bd}{\begin{definition}}
	\newcommand{\ed}{\end{definition}}

\DeclareMathOperator{\ran}{ran}

\DeclareMathOperator{\dist}{dist}
\DeclareMathOperator{\didistance}{didist}
\newcommand{\mc}{\mathcal}
\newcommand{\mb}{\mathbf}
\newcommand{\bs}{\mathbf{\Sigma}^1_1}

\newcommand{\bbo}{\mathbf{\Delta}^1_1}
\newcommand{\bp}{\mathbf{\Pi}^1_1}

\newcommand{\distance}[3]{\ifthenelse{\isempty{#3}}{\dist(#1,#2)}{\dist^{#3}(#1,#2)}}
\newcommand{\didist}[3]{\ifthenelse{\isempty{#3}}{\didistance(#1,#2)}{\didistance^{#3}(#1,#2)}}
\newcommand{\digraph}[3]{\ifthenelse{\equal{#1}{b}}{\mathbb{#2}_{#3}}
	{{#2}_{#3}}}
\newcommand{\linegraph}[3]{\ifthenelse{\equal{#1}{b}}{\mathbb{#2}_{#3}}
	{#2_{#3}}}

\newcommand{\underlyingspace}[3]{\ifthenelse{\equal{#1}{b}}{\mathbb{#2}_{#3}}
	{#2_{#3}}}
\newcommand{\distanceset}[2]{\ifthenelse{\isempty{#2}}{D(#1)}{D^{#2}(#1)}}

\newcommand{\concatt}{%
	\mathbin{\raisebox{1ex}{\scalebox{.7}{$\frown$}}}%
}

\begin{document}

	\thanks{}
	
	
	
	\title[]{Complexity of Linear Equations and Infinite Gadgets}
	
\author{Jan Greb\'ik}

\address{Faculty of Informatics, Masaryk University, Botanicka 68A, 60200 Brno, Czech Republic}

\author{Zolt\'an Vidny\'anszky}
\address{E\"otv\"os Lor\'and University, Institute of Mathematics, P\'azm\'any P\'eter stny. 1/C, 1117 Budapest, Hungary}

	
	
	\maketitle
	
	\begin{abstract}

    We investigate the descriptive set-theoretic complexity of the solvability of a Borel family of linear equations over a finite field.
    Answering a question of Thornton, we show that this problem is already hard, namely $\mathbf{\Sigma}^1_2$-complete.
    This implies that the split between easy and hard problems is at a different place in the Borel setting than in the case of the CSP Dichotomy.

	\end{abstract}

	\maketitle
	\section{Introduction}
	
	One of the most prominent directions of modern descriptive set theory is descriptive graph combinatorics, that is, considering definable versions of finite graph-theoretic notions on definable graphs (see, e.g., \cite{marks2022measurable,kechris_marks2016descriptive_comb_survey, pikhurko2021descriptive_comb_survey}). In particular, Borel colorings of Borel graphs play an important role in this area. Recall that a \emph{Borel graph} $\mc{G}$ on a Borel space $X$ is a symmetric Borel subset of $X^2$ disjoint from the diagonal. A \emph{Borel coloring of $\mc{G}$} for $n \in \{2,3,\dots,\aleph_0\}$ with $n$-many colors is a proper $n$-coloring so that all the color classes are Borel, i.e., a Borel map $c:X \to n$ so that if $(x,y) \in \mc{G}$ then $c(x) \neq c(y)$.
	
	It has been shown in \cite{todorvcevic2021complexity} (see also \cite{brandt2021homomorphism}) that deciding Borel $n$-colorability of a Borel graph is as hard as it gets, namely $\mathbf{\Sigma}^1_2$-complete for any finite $n \geq 3$ (see Section \ref{s:pr} for the definitions). In contrast, deciding Borel $2$-colorability is simpler, it is $\mathbf{\Pi}^1_1$, see, e.g., \cite{benen}. 
	
	Clearly, a Borel $n$-coloring is the same as a Borel homomorphism to $K_n$, the complete graph on $n$-vertices. Thus, it makes sense to consider the complexity of the Borel homomorphism problem to a given finite graph, or more generally, to a finite relational structure.
	
	Let $L$ be a relational language with relations $(R_i)_{i \leq k}$, having arity $n_i$. A \emph{Borel $L$-structure} $\mc{G}$ on the Borel (underlying) space $X$ is a  collection of Borel sets $R^\mc{G}_i \subseteq X^{n_i}$. A \emph{Borel homomorphism} between $L$-structures is a Borel map between the underlying spaces, which maps $R_i$-related tuples to $R_i$-related tuples for each $i$. Note that finite $L$-structures are also Borel with the trivial Borel structure. 
	
	Thornton in \cite{thornton2022algebraic} has initiated a systematic study of the complexity of homomoprhism problems of the following form: 
	\[\text{Given a finite $L$-structure $H$ and a Borel $L$-structure $\mc{G}$,}\]\[\text{does $\mc{G}$ admit a Borel homomorphism to $H$?}\] This is motivated by the deep theory (see, e.g., \cite{brady2022notes}) of constraint satisfaction problems (CSPs) in computational complexity and the recent classification theorem of Bulatov \cite{bulatov2017dichotomy} and Zhuk \cite{zhuk2020proof}, which states that a homomorphism problem is either in P or NP-complete (of course, this is not known to be exclusive). Moreover, there is a complete algebraic understanding, in terms of $H$, of the homomorphism problems that are going to be easy.
	
	Based on the distinction between $2$ and $3$ colorability in the Borel context mentioned above, one might hope to detect the same split between easy and hard problems. Since in the Borel context it can be shown that the classes $\mathbf{\Pi}^1_1$ and $\mathbf{\Sigma}^1_2$ do not coincide, this would be even more intriguing (note that this would give no information about finite separation, nevertheless). To this end, based on the complexity result in \cite{todorvcevic2021complexity} and the algebraic theory in \cite{krokhin2005complexity}, Thornton has proved that if the $H$-homomorphism is known to be NP-complete then the corresponding Borel version is $\mathbf{\Sigma}^1_2$-complete. 

	In this paper, however, we show that the hoped analogy between Borel and classical CSPs fails.
    Namely, we show that solving systems of linear equations, which can be done using Gaussian elimination in polynomial time in the classical CSP, is already difficult in the Borel context.

    A system of linear equations over a finite field $\mathbb{F}$ is easily seen to be equivalent to one where every equation contains $\leq n$-many variables (this can be done by introducing dummy variables, see, e.g., \cite[Section 2]{thornton2022algebraic}). Such systems in turn can be encoded by homomorphisms to a finite structure, with universe $\mathbb{F}$ and $n$-ary relations encoding the value of the sums of the given three elements of the field. For example, for $\mathbb{F}_2$ and equations of size $\leq 3$, the resulting finite structure is a structure on $\{0,1\}$ with two trenary relations $R_0$ and $R_1$ so that $(x_0,x_1,x_2) \in \{0,1\}^3$ is in relation $R_i$ iff $x_0+x_1+x_2=i$. Hence solving systems of linear equations are particular instances of homomorphism problems. 
    
    Clearly, as mentioned above, by Gaussian elimination, these are easy in the finite case. In contrast, we show the following. 
	
	\begin{theorem}
		\label{t:main}
		Deciding the solvability of Borel systems of linear equations over a finite field $\mathbb{F}$ is $\mathbf{\Sigma}^1_2$-complete. 
	\end{theorem}

	Our proof has two steps. First, we consider coloring problems of countably infinite dimensional hypergraphs, which turn out to be already complicated to decide (Theorem \ref{t:hypercomplex}). We believe that this result could be interesting on its own. Second, based on the graph $\mathbb{G}_0$ of Kechris-Solecki-Todor\v{c}evi\'c \cite{KST}, we construct infinite gadgets, which allow us to reduce the hypergraph coloring problems to solving systems of linear equations. 
	
	\subsection*{Roadmap} Section \ref{s:pr} contains the necessary technical preliminaries. In Section \ref{s:hypercomplex} we show the complexity of weak version of $2$-colorings for infinite dimensional hypergraphs, while Section \ref{s:gadget} contains the gadget construction. In all the cases, we prove the results for $\mathbb{F}=\mathbb{F}_2$ first and then indicate what to modify for general finite fields. Finally, in Section \ref{s:problems} we formulate some open problems.
	
	\subsection*{Acknowledgments} 
	We are very grateful to Michael Pinsker and Riley Thornton for their insightful comments.
    ZV was supported by Hungarian Academy of Sciences Momentum Grant no. 2022-58 and National Research, Development and Innovation Office (NKFIH) grants no.~113047, ~129211.

    \section{Preliminaries}
       \label{s:pr}

    \subsection{Coding.} Borel, analytic, co-analytic and projections of co-analytic sets are denoted by $\mb{\Delta}^1_1$, $\mb{\Sigma}^1_1$, $\mb{\Pi}^1_1, \mb{\Sigma}^1_2$, see \cite{kechrisclassical} for the basic results about these classes.
    
    First we need to fix an encoding of Borel sets. Let $X$ be a Polish space there are sets $\mb{BC}(X)$, $\mb{A}(X)$, and $\mb{C}(X)$ with properties summarized below.

    \begin{proposition} (see \cite[3.H]{moschovakis2009descriptive})
    	\label{f:prel}
    	\begin{itemize}
    		\item $\mb{BC}(X) \in \bp(\oom)$, $\mb{A}(X) \in \bs(\oom \times X)$, $\mb{C}(X) \in \bp(\oom \times X)$,
    		\item for $c\in \mb{BC}(X)$ and $x \in X$ we have $(c,x) \in \mb{A}(X) \iff (c,x) \in \mb{C}(X)$,
    		\item if $P$ is a Polish space and $B \in \bbo(P \times X)$ then there exists a Borel map $f:P \to \oom$ so that $ran(f) \subset \mb{BC}(X)$ and for every $p \in P$ we have $\mathbf{A}(X)_{f(p)}=B_p$.
    	\end{itemize}
    	
    \end{proposition} 
    
    Let $\mathbf{\Delta}$ be a family of subsets of Polish spaces. Recall that a subset $A$ of a Polish space $X$ is \emph{$\mathbf{\Delta}$-hard,} if for every $Y$ Polish and $B \in \mathbf{\Delta}(Y)$  there exists a continuous map $f:Y \to X$ with $f^{-1}(A)=B$. A set is \emph{$\mathbf{\Delta}$-complete} if it is $\mathbf{\Delta}$-hard and in $\mathbf{\Delta}$. 
        
    Now we can define what we mean by some collection being $\mathbf{\Sigma}^1_2$-complete.
    
    \begin{definition}
    	Let $\mc{C}$ be a collection of Borel objects and $\mb{\Gamma}$ be a family of sets. We say that \emph{$\mc{C}$ is in $\mb{\Gamma}$} if for any Polish space $X$ we have $\mb{BC}(X)\cap \mc{C}$ is in $\mb{\Gamma}$. If for some $X$ the set is $\mb{\Gamma}$-hard, then $\mc{C}$ is said to be $\mb{\Gamma}$-complete.
    \end{definition}
	
	 In all of our considerations, the class $\mc{C}$ is going to be invariant under Borel isomorphisms of the underlying space, hence we can forget about the existential quantifier in the above definition. Also, note that in the case of $\mathbf{\Sigma}^1_2$-hardness, which is our main interest, the existence of a continuous map $f$ above can be relaxed to the existence of a Borel map (in fact, to something even weaker; see \cite{sabok}).
		
	A slight technical strengthening of this notion is the following. In practice, we will always get this stronger variant. 
	
	\begin{definition}
		Let $\mc{C}$ be a collection of Borel objects and $\mb{\Gamma}$ be a family of sets. We say that \emph{$\mc{C}$ is uniformly $\mb{\Gamma}$-hard} if there are Polish spaces $Z$ and $X$ and a Borel set $B \subseteq Z \times X$ so that $\{z:B_z \in \mc{C}\}$ is $\mb{\Gamma}$-hard. 
	\end{definition}

    \subsection{Borel CSPs.}
    
    We will need a more general class of structures than the ones defined in the introduction (i.e., we want to allow relations of infinite arities). Let $L$ be a signature consisting of relations $(R_i)_{i}$ of finite or countably infinite arities $(r_i)_{i}$, respectively. A \emph{Borel (relational) structure $\mc{G}$} on the space $X$ is a collection of Borel subsets $R_i^\mc{G} \subseteq X^{r_i}$. If $\mc{G}$ and $\mc{H}$ are Borel relational structures with the same signature, on spaces $X$ and $Y$, a \emph{Borel homomorphism from $\mc{G}$ to $\mc{H}$} is a Borel map $\phi: X \to Y$ so that $\forall i \ \forall x \in X^{r_i} \ (x \in R^\mc{G}_i \implies \phi(x) \in R^\mc{H}_i)$, where $\phi(x)$ is the sequence obtained by applying $\phi$ to $x$ elementwise. If $B$ is a Borel subset of $X$, the \emph{restriction} of $\mc{G}$ to $B$ is the structure on $B$ with the relations $B^{r_i} \cap R^\mc{G}_i$. 
    
    For a Borel structure $\mathcal{H}$, the \emph{Borel homomorphism problem to $\mathcal{H}$} is denoted by $CSP_B(\mathcal{H})$. This problem is said to be in \emph{$\mb{\Gamma}$} if the codes of the structures that admit a Borel homomorphism to $\mc{H}$ form a set in $\mb{\Gamma}$, and we apply a similar convention to being $\mb{\Gamma}$-hard, etc.
    
    The easy upper bound on the complexity is $\mathbf{\Sigma}^1_2$ (see \cite{frisch2024hyper}):
    
    \begin{proposition}
    	\label{pr:homo}
    	$CSP_B(\mc{H})$ is $\mathbf{\Sigma}^1_2$, for any Borel structure $\mathcal{H}$.
    \end{proposition}
	It is convenient to define a notion of reduction between Borel CSPs. Let $L$ be a language with relations $(R_i)_i$, having arities $r_i$. Let $Y=\bigsqcup_i X^{r_i}$. For a Borel space $Z$, a \emph{$Z$-parametrized Borel family of $L$-structures} on the space $X$ is a Borel subset of $Z \times Y$. For $z \in Z$ the set $B_z$ is interpreted as an $L$-structure where $R^{B_z}_i=X^{r_i} \cap B_z$. 
	\begin{definition}
		Let $L$ and $L'$ be relational languages and $\mc{H}$ and $\mc{H}'$ Borel structures in the respective languages. We say that $CSP_B(\mc{H})$ \emph{uniformly Borel reduces to} $CSP(\mc{H}')$ if for any $Z$-parametrized Borel family $B$ of $L$-structures, there is a $Z$-parametrized Borel family of $L'$-structures $B'$ so that $B_z$ admits a Borel homomorphism to $\mc{H}$ iff $B'_z$ admits a Borel homomorphism to $\mc{H}'$.
		
	\end{definition} 
	
	The following is clear from the definitions and Proposition \ref{f:prel}. 
	\begin{proposition}
		\label{p:prel}
		Assume that $CSP_B(\mc{H})$ uniformly Borel reduces to $CSP_B(\mc{H}')$ and that $CSP_B(\mc{H})$ is uniformly $\mb{\Sigma}^1_2$-hard. Then so is $CSP_B(\mc{H}')$.
		
	\end{proposition}
	
	If $\mathbb{F}$ is a finite field, $CSP_B(\mathbb{F})$ will stand for the homomorphism problem encoding systems of linear equations over $\mathbb{F}$, as has been described in the introduction. Nevertheless, we will use the more intuitive description in terms of equations when we talk about these. Moreover, if the equations are defined on the underlying space $Y$, we will denote the variable corresponding to an element $y\in Y$ by $x_y$. 
	
	We will also use the term \emph{$\N$-dimensional Borel hypergraph} for a Borel $L$-structure, where $L$ contains a single relation of arity $\aleph_0$.

    \subsection{Complexity.} The most important tool (and essentially the only tool which we are aware of) to establish complexity results of this type is the usage of non-dominating sets. Recall that $[\N]^\N$ is the collection of infinite subsets of the natural numbers. We identify elements of $[\N]^\N$ with their increasing enumeration. If $x,y \in [\mathbb{N}]^{\mathbb{N}}$ let us use the notation $y \leq^\infty x$ in the case the set $\{n:y(n) \leq x(n)\}$ is infinite and $y \leq^* x$ if it is co-finite. Set \[\mathcal{D}=\{(x,y):y \leq^\infty x\}.\]
    
    The following is a consequence of a general theorem in \cite{todorvcevic2021complexity} and has been isolated in \cite{frisch2024hyper}. 
    \begin{theorem}
    	\label{t:complexity}
    	Let $\mathcal{H}$ be a Borel structure on some Polish space $X$ and assume that there exists a Borel structure $\mathcal{G}$ on $[\N]^\N$ that does not admit a Borel homomorphism to $\mathcal{H}$ and a Borel map $\Psi:\mathcal{D} \to X$ so that for each $f$ we have that $\Psi_f$ is a homomorphism from $\restriction{\mathcal{G}}{\mathcal{D}_f}$ to $\mathcal{H}$. Then the Borel structures which admit a Borel homomorphism to $\mathcal{H}$ form a $\mathbf{\Sigma}^1_2$-complete set, in fact, $CSP_B(\mc{H})$ is uniformly $\mathbf{\Sigma}^1_2$-hard.  
     \end{theorem}

 	\section{Eventually weakly alternating $2$-colorings}
 	\label{s:hypercomplex}
 	
 	The first step of our proof is to show that the homomorphism problem to a certain infinite dimensional Borel hypergraph is $\mathbf{\Sigma}^1_2$-complete.
 	Essentially, the hypergraph will encode $2$-colorings, which eventually alternate along every hyperedge, in the following sense.
 	\begin{definition}
 		Let $x \in 2^\N$. We say that $x$ \emph{eventually weakly alternates}, if there is some $i_0$, such that for all $i \geq i_0$ we have $x(2i) \neq x(2i+1)$.
 		
 		Let $H_0$ be the $\mathbb{N}$-dimensional Borel hypergraph on the set $\{0,1\}$ containing the hyperedges that eventually weakly alternate. 
 	\end{definition}

 	Now we have the following theorem.
 	
 	\begin{theorem}
 		\label{t:hypercomplex}
 		The collection of $\N$-dimensional Borel hypergraphs admitting a homomorphism to $H_0$ is $\mathbf{\Sigma}^1_2$-complete. In fact, $CSP_B(H_0)$ is uniformly $\mathbf{\Sigma}^1_2$-hard. 
 	\end{theorem}
 	
 	In order to show this theorem we will use Theorem \ref{t:complexity}. We will also make a critical use of the \emph{shift-map} $S$ on $[\N]^\N$, that is, the map defined by
 	\[S(x)=x \setminus \{x(0)\}.\]
 	 Define a hypergraph $\mathcal{H}$ on $[\N]^\N$ as follows. 
 	\begin{definition}
 		For $x \in [\N]^\N$ let $H_x$  be a hyperedge defined as 
 		\[H_x(2i)=\{x(0),x(1)\} \cup S^{i+2}(x),\]
 		and 
 		\[H_x(2i+1)=\{x(1)\} \cup S^{i+2}(x).\]
 		Let $\mathcal{H}=\{H_x:x \in [\N]^\N\}$. 
 	\end{definition}
 		Observe that for each $i$ and $x$ we have $H_x(2i+1)=S(H_x(2i))$.
 	\begin{proposition}
 		\label{pr:nohomo}
 		$\mathcal{H}$ does not admit a Borel homomorphism to $H_0$. 
 	\end{proposition}
 	\begin{proof}
 		If $h:[\N]^\N \to 2$ was a Borel homomorphism then $h$ is constant on a set of the form $[x]^\N$ by the Galvin-Prikry theorem (see \cite[Section 19.C]{kechrisclassical}), but $ran(H_x) \subset [x]^\N$, a contradiction. 
 	\end{proof}
 
 	\begin{proposition}
 		\label{pr:yeshomo} 
 		There is a Borel map $\Psi:\mc{D} \to 2$ such that for each $f \in [\N]^\N$ we have that $\Psi_f$ is a Borel homomorphism from $\restriction{\mc{H}}{\mc{D}_f}$ to $H_0$.
 	\end{proposition}
 	\begin{proof}
 		Given $f$, we will construct the desired homomorphism $\Psi_f$ and it will be clear from the construction that $\Psi$ is Borel. The following claim is immediate.
 		
 		\begin{claim}
 			\label{cl:easy}
 			Let $f'(n)=f(n^2)$ for $n \geq 1$, and $f'(0)=0$. Then for all $x \in \mc{D}_f$ there exists an $n$ such that $|x\cap [f'(n),f'(n+1))|\geq 3$. 
 		\end{claim}	
 	
 		Let \[A=\{x\in  \mc{D}_f: |x\cap [f'(n),f'(n+1))|=3,\]\[ \text{ where $n$ is minimal with $|x\cap [f'(n),f'(n+1))| \neq \emptyset$}\}.\] 
 		Now we define the mapping $\Psi_f:\mc{D}_f \to 2$ by letting 
 			\[
 		\Psi_f(x)=
 		\begin{cases}
 		1, & \text{if the minimal $l \geq 0$ with }S^l(x) \in A \text{ is even}\\
 		0, & \text{otherwise.}
 		\end{cases}
 		\]
 		It follows from Claim \ref{cl:easy} that $h(x)$ is well-defined for each $x \in \mc{D}_f$. Let us make an easy observation. 
 		\begin{claim}
 			\label{cl:immediate}
 			If $y,S(y)$ are not elements of $A$, then $h(y) \neq h(S(y))$. 
 		\end{claim}
 			
 		The next lemma will be sufficient to finish the proof of Proposition \ref{pr:yeshomo}.
 		\begin{lemma}
 			\label{l:inter}
 			For each $x \in [\N]^\N$ the set $A \cap \ran(H_x)$ is finite.  
 		\end{lemma}	
 	
 		\begin{proof}
 			Let $n$ be such that $x(0),x(1) < f'(n)$. Then there is an $l$ with $S^l(x)(0) \not \in [0,f'(n))$. Then for every $j$ with $j/2 > l-2$, we have that $1 \leq |H_x(j) \cap [0,f'(n))|\leq 2$, since this set is either only $\{x(1)\}$ or $\{x(0),x(1)\}$. In particular, the first interval determined by $f'$ which intersects $H_x(j)$ contains at most these two elements, so $H_x(j) \not \in A$. 
 		\end{proof}
 			Thus, the combination of Lemma \ref{l:inter} and Claim \ref{cl:immediate}
 			yields that $\Psi_f$ is a homomoprhism to $H_0$. 		
 	\end{proof}
 	\begin{proof}[Proof of Theorem \ref{t:hypercomplex}]
 		Now, Theorem \ref{t:hypercomplex} follows from Propositions \ref{pr:nohomo}, \ref{pr:yeshomo} and Theorem \ref{t:complexity}.
 	\end{proof}

 	\subsection{Fields of characteristic $p$}
 	
 	\label{ss:pchar}
 	A straightforward modification of the result in this section will be needed for general finite fields. Fix a field $\mathbb{F}$ of characteristic $p$.

 	\begin{definition}
 		Let $x \in \mathbb{F}^\N$. We say that $x$ \emph{eventually weakly $p$-alternates}, if there is some $i_0$, such that for all $i \geq i_0$ we have that the sequence \[\sum_{0 \leq j \leq p-1} x(pi+j)=1.\]
 		
 		Let $H^p_0$ be the $\mathbb{N}$-dimensional Borel hypergraph on the set $\mathbb{F}$ containing the hyperedges that eventually weakly $p$-alternate. 
 	\end{definition}

 	Now we have the following theorem.
 	
 	\begin{theorem}
 		\label{t:hypercomplexp}
 		The collection of $\N$-dimensional Borel hypergraphs admitting a homomorphism to $H^p_0$ is $\mathbf{\Sigma}^1_2$-complete. In fact, $CSP_B(H_0)$ is uniformly $\mathbf{\Sigma}^1_2$-hard. 
 	\end{theorem}
 	
 	Let us quickly indicate how to show this theorem. The corresponding hypergraph $\mathcal{H}^p$ is defined as follows:
 	\begin{definition}
 		For $x \in [\N]^\N$ let $H_x$  be a hyperedge defined as 
 		\[H_x(pi+j)=\{x(j),\dots,x(p-1)\} \cup S^{i+p}(x),\]
 		
 		and $\mc{H}^p=\{H_x:x \in [\N]^\N\}.$
 	\end{definition}
 	The analogue of Proposition \ref{pr:nohomo} clearly holds. As for Proposition \ref{pr:yeshomo}, one defines the set
 	\[A_p=\{x\in  \mc{D}_f: |x\cap [f'(n),f'(n+1))|=p+1,\]\[ \text{ where $n$ is minimal with $|x\cap [f'(n),f'(n+1))| \neq \emptyset$}\}.\] 
 	and the mapping $\Psi_f:\mc{D}_f \to 2$ by letting
 	\[
 	\Psi_f(x)=
 	\begin{cases}
 	1, & \text{if the minimal $l$ with }S^l(x) \in A \text{ is divisible by $p$}\\
 	0, & \text{otherwise.}
 	\end{cases}
 	\]
 	It is easy to check that this works.

    \section{The gadgets}
   \label{s:gadget}
   In this section we prove that the previously defined hypergraph coloring problems can be encoded using systems of linear equations. 
   \begin{theorem}
   	\label{t:reduction}
   	To each $\N$-dimensional Borel hypergraph $\mathcal{H}$ we can associate a Borel system of linear equations $A_\mc{H}$ over $\mathbb{F}_2$ so that $A_\mc{H}$ admits a Borel solution, if an only if $\mathcal{H}$ admits a Borel homomorphism to $H_0$. In fact, $CSP_B(H_0)$ uniformly reduces to $CSP_B(\mathbb{F}_2)$. 
   \end{theorem}
   
   In order to prove this theorem, we will need an infinitary gadget. This relies on a strong failure of compactness in the Borel case. 
   
   \begin{proposition}[Gadget construction]
   	\label{p:all}
   	There exists a Borel system of linear equations over $\mathbb{F}_2$ on some space $A$, denoted by $\mathcal{A}_0$, which has a distinguished collection $(x_i)_{i \in \N}$ of variables so that a map $\phi:\{x_i: i \in \N\} \to \mathbb{F}_2$ can be extended to a Borel solution to $\mathcal{A}_0$ if and only if $(\phi(x_i))_i$ eventually weakly alternates. 
   	
   	Moreover, this can be done in a uniform manner, that is, there exists a Borel map $\bar{\phi}:2^\N \times A \to \mathbb{F}_2$, so that if $z=(\phi(x_i))_i$ eventually weakly alternates then $\bar{\phi}_z$ extends $(\phi(x_i))_i$ to a solution. 
   \end{proposition}

   The below construction is a version of the construction of the graph $\mathbb{G}_0$ from the seminal paper of Kechris-Solecki-Todor\v{c}evi\'c \cite{KST}. Essentially, every vertex of $\mathbb{G}_0$ will serve as a variable, and to edges we associate pairs of distinguished variables. If $\phi(x_{2i}) \neq \phi(x_{2i+1})$, then the two variable along all corresponding edges must be equal, if not, they must be distinct. A single $\mathbb{G}_0$-type copy would ensure that there is no Borel homomorphism $\phi$, in which there are only finitely elements $i$ with $\phi(x_{2i})\neq \phi(x_{2i+1})$, and we construct one for every possible infinite subset of $i$'s (the first coordinate below).
   
   \begin{proof}[Proof of Proposition \ref{p:all}]
   	Fix a sequence $s_n \in 2^n$ which is \emph{dense}, that is, for every $t \in 2^{<\N}$ there exists an $n$ with $t \sqsubset s_n$. The underlying space $A$ of the systems of linear equations is going to be $\left([\N]^\N \times 2^\N\right) \sqcup \N$, where $\N$ will be the collection of distinguished variables.
   	
   	For $r \in [\N]^\N$, $u,v \in 2^\N, j \in 2, i \in \N$ add the equation \begin{equation}\label{eq} x_{(r,u)}+x_{(r,v)}+x_{2r(i)}+x_{2r(i)+1}=1\end{equation} if $u=s_i \concatt (j) \concatt c$ and $v=s_i \concatt (1-j) \concatt c$ for some $c \in 2^\N$. 
   	
   	\begin{claim}
   		\label{cl:firstno}
   		Assume that $\phi:\{x_i:i \in \N\} \to 2$ is a map so that there is an infinite set of $i$'s with $\phi(x_{2i})=\phi(x_{2i+1})$. Then $\phi$ does not extend to a Borel solution.  
   	\end{claim}
   	\begin{proof}
   		Let $r$ be the set of such $i$s, and assume that $\bar{\phi}$ is a Borel extension. Then there exists a $t \in 2^{<\N}$ and an $j \in \{0,1\}$ so that $N_t \setminus \bar{\phi}^{-1}(j)_{r}$ is meager in $N_t$. By density, there is an $i$ with $t \sqsubseteq s_i$. It follows by a standard Baire category argument that there are $u,v \in N_{s_i}$ and $c \in 2^\N$ such that $\bar{\phi}(x_{(r,u)})=\bar{\phi}(x_{(r,v)})$ and $u=s_i \concatt (0) \concatt c$ and $v=s_i \concatt (1) \concatt c$. But $\bar{\phi}(x_{2r(i)})=\bar{\phi}(x_{2r(i)+1})$, contradicting \eqref{eq}. 
   	\end{proof}
   	
   	\begin{claim}
   		\label{cl:firstyes} Assume that $\phi:\{x_i:i \in \N\} \to 2$ is a map so that $(\phi(x_i))_{i \in \N}$ eventually weakly alternates. Then $\phi$ extends to a Borel solution. 
   	\end{claim}
   	\begin{proof}
   		Let $i_0$ be such that $\phi(x_{2i})\neq\phi(x_{2i+1})$ for all $i \geq i_0$. We will commit to defining an extension which is constant on sets of the form $\{r\} \times N_t$ whenever $t \in 2^{i_0}$, and only depends on $r(j)$ for $j \leq i_0$, ensuring that the resulting map is continuous. Note that this commitment is consistent with 
   		\eqref{eq}, since $r(i) \geq i$. 
   		
   		Now, define a graph $G$ on $2^{i_0}$ by connecting $u,v$ if there is an $n$ with $u=s_n \concatt (\varepsilon) \concatt c$ and $v=s_n \concatt (1-\varepsilon) \concatt c$ with $c \in 2^{i_0-n-1}$. An easy induction shows that this graph is acyclic. 
   		
   		Fix an $r \in [\N]^\N$. Starting from a vertex of $G$, coloring inductively the neighbors, we can define a coloring $\psi_r: 2^{i_0} \to 2$ so that for $(u,v) \in 2^{i_0}$ we have \begin{equation}\label{eq2}\psi_r(u)=\psi_r(v) \iff \phi(x_{2r(i)}) \neq \phi(x_{2r(i)+1}),\end{equation}
   		where $u=s_i \concatt (\varepsilon) \concatt c$ and $v=s_i \concatt (1-\varepsilon) \concatt c$. Clearly, there is exactly one such an $i$, thus from acyclicity of $G$ we get that such a coloring can be constructed.  
   		
   		Finally, for $u \in 2^\N$ let $\bar{\phi}(x_{(r,u)})=\psi_r(\restriction{u}{i_0})$. It is clear that $\bar{\phi}$ is continuous, and we check that it is indeed a solution: indeed, assume that \[x_{(r,u)}+x_{(r,v)}+x_{2r(i)}+x_{2r(i)+1}=1\] is in $\mathcal{A}_0$ with $u=s_i \concatt (0) \concatt c$ and $v=s_i \concatt (1) \concatt c$. If $i>i_0$, then \[\bar{\phi}(x_{(r,u)})=\psi_r(\restriction{u}{i_0})=\psi_r(\restriction{v}{i_0})=\bar{\phi}(x_{(r,v)}),\] so together with $\phi(x_{2r(i)})\neq\phi(x_{2r(i)+1})$ we have \eqref{eq}.
   		In case $i \leq i_0$ \eqref{eq2} guarantees that  \eqref{eq} holds.

   	\end{proof}
   	
   	It is clear from the construction of the coloring in Claim \ref{cl:firstyes} that it is performed in a uniform manner in $z=(\phi(x_i))_{i}$, in fact the extension only depends on the first $2i_0$ coordinates $z$, yielding the second claim of Proposition \ref{p:all}.
   \end{proof}
   
   All that is left to prove Theorem \ref{t:reduction} is to glue $\mathcal{A}_0$ type gadgets along each hyperedge, and check that this construction works.
   
   \begin{proof}[Proof Theorem \ref{t:reduction}]
   	Let $\mathcal{H}$ be an $\N$-dimensional Borel hypergraph on a space $X$, and identify $\mc{H}$ with its edge set $\mc{H} \subseteq X^\N$. Let $\mathcal{A}_0$ be the system defined in Proposition \ref{p:all} on the space $A=Y_0 \sqcup \N$. We define a Borel system of linear equations $\mathcal{A}_H$ on the space $X \sqcup X^\N \times A (=X\sqcup (X^\mathbb{N}\times Y_0)\sqcup (X^\mathbb{N}\times \mathbb{N}))$ by adding for all $H\in \mc{H}$ the equation $\sum x_{(H,a_j)}=1$, whenever $\sum x_{a_j}=1$ is in $\mathcal{A}_0$, and moreover,  for all $a \in X$
   	\begin{equation}
   	x_{a}=x_{(H,i)}  \text{, if } a=H(i). \label{eq3}
   	\end{equation}
   	
   	Assume that $\mc{A}_\mc{H}$ has a Borel solution $\phi$. Then $\restriction{\phi}{X}$ is a Borel map to $\{0,1\}$ and if $H \in \mc{H}$ then $(\phi(x_{H(i)}))_i$ must be eventually weakly alternating: if it was not, then $\restriction{\phi}{\{H\} \times A}$ would be a solution of $\mathcal{A}_0$ where the distinguished variables, $(x_{(H,i)})_{i}$ get values that are not eventually alternating, contradicting Proposition \ref{p:all}. Thus, $\mc{H}$ admits a Borel homomorphism to $H_0$. 
   	
   	Conversely, if $\mc{H}$ admits a Borel homomorphism $\psi$ to $H_0$, then we can define a Borel solution $\phi$ to $\mc{A}_H$ by letting
   	$\phi(x_a)=\psi(a)$ and $\phi(x_{(H,y)})=\bar{\phi}_{z}(y)$, where $z=\psi(H(i))_{i \in \N}$ and the map $\bar{\phi}_z$ is from Proposition \ref{p:all}. Since this map is Borel, the resulting $\phi$ is going to be Borel as well. By its definition, $\phi$ obeys all the equations coming from $\mathcal{A}_0$, and moreover \eqref{eq3} is satisfied, as whenever $a=H(i)$ and $z=(\psi(H(i)))_{i}$ we have
   	\[\phi(x_a)=\psi(a)=\psi(H(i))=z(i)=\bar{\phi}_z(x_i)=\phi(x_{(H,i)}).\]
   	
   	It is clear from the construction of $\mathcal{H} \mapsto \mathcal{H}_\mathcal{A}$ is a uniform Borel reduction. 
   \end{proof}
   Now, a combination of Theorem \ref{t:hypercomplex} and Theorem \ref{t:reduction} yields our main result for $\mathbb{F}_2$:
   \begin{corollary} $CSP_{B}(\mathbb{F}_2)$ is uniformly $\mathbf{\Sigma}^1_2$-complete.  
   \end{corollary}
   
   \subsection{Fields of characteristic $p$ and the main theorem.}
   Fix again a finite field $\mathbb{F}$ of characteristic $p$. We indicate the required modifications.
   \begin{theorem}
   	\label{t:reductionp}
   	To each $\N$-dimensional Borel hypergraph $\mathcal{H}$ we can associate a Borel system of linear equations $A_\mc{H}$ over $\mathbb{F}$ so that $A_\mc{H}$ admits a Borel solution, if an only if $\mathcal{H}$ admits a Borel homomorphism to $H^p_0$. In fact, $CSP_B(H^p_0)$ uniformly Borel reduces to $CSP_B(\mathbb{F})$. 
   \end{theorem}
   
   The gadget now looks the following way. 
   
   \begin{proposition}[Gadget construction]
   	\label{p:allp}
   	There exists a Borel system of linear equations over $\mathbb{F}$, $\mathcal{A}^p_0$ which has a distinguished collection $(x_i)_{i \in \N}$ of variables so that a map $\phi:\{x_i: i \in \N\} \to \mathbb{F}$ can be extended to a Borel solution to $\mathcal{A}^p_0$ if and only if $(\phi(x_i))_i$ eventually weakly $p$-alternates. 
   	
   	Moreover, this can be done in a uniform manner, that is, there exists a Borel map $\bar{\phi}:2^\N \times A \to \mathbb{F}$, so that if $z=(\phi(x_i))_i$ eventually weakly $p$-alternates then $\bar{\phi}_z$ extends $(\phi(x_i))_i$ to a solution. 
   \end{proposition}

   \begin{proof}[Proof of Proposition \ref{p:allp}]
   	Fix a sequence $s_n \in p^n$ which is \emph{dense}, that is, for every $t \in p^{<\N}$ there exists an $n$ with $t \sqsubset s_n$. The underlying space of the systems of linear equations is going to be $[\N]^\N \times p^\N \sqcup \N$, and for $r \in [\N]^\N$, $u_j \in p^\N$ add the equation \begin{equation}\label{eqp}\sum_{0 \leq j \leq p-1} x_{(r,u_j)}+ \sum_{0 \leq j \leq p-1}x_{pr(i)+j}=1\end{equation} if $u_j=s_i \concatt (j) \concatt c$ for some $c \in p^\N$ and $0 \leq j \leq p-1$. By $\mathbb{F}$ having characteristic $p$, this is implies that  \[ (\forall j, j' \ x_{(r,u_j)} = x_{(r,u_{j'})})\implies   \sum_{0 \leq j \leq p-1}x_{pr(i)+j}=1.\]
   	
   	Using this equation, similarly to the proof of Claim \ref{cl:firstno} we get:
   	\begin{claim}
   		\label{cl:firstnop}
   		Assume that $\phi:\{x_i:i \in \N\} \to \mathbb{F}$ is a map so that there is an infinite set of $i$'s with $\sum_j \phi(x_{pi+j}) \neq 1$, i.e., $(\phi(x_{i}))_{i}$ is not weakly eventually $p$-alternating. Then $\phi$ does not extend to a Borel solution.  
   	\end{claim}
   	
   	Finally, let us check the analogue of Claim \ref{cl:firstyes}. This is essentially the only place where something slightly different is needed. 
   	
   	\begin{claim}
   		\label{cl:firstyesp} Assume that $\phi:\{x_i:i \in \N\} \to \mathbb{F}$ is a map so that $(\phi(x_i))_{i \in \N}$ eventually weakly alternates. Then $\phi$ extends to a Borel solution of $\mathcal{A}^p_0$. 
   	\end{claim}
   	\begin{proof}
   		Let $i_0$ be such that $\sum_{0 \leq j \leq p-1}x_{pr(i)+j}=1$ for all $i \geq i_0$. Define a hypergraph $G$ on $p^{i_0}$ so that $(u_j)_{j \leq p-1}$ forms a hyperedge if there is an $n$ with $u_j=s_n \concatt (j) \concatt c$ with $c \in p^{i_0-n-1}$. An easy induction shows that this hypergraph is Berge acyclic as well, i.e., its incidence graph is acyclic.
        
   		\begin{subclaim}
   			\label{cl:solution}
   			Assume that $A$ is a finite system of linear equations over $\mathbb{F}$ so that the hypergraph on the variables where the edges correspond to the equations is acyclic. Then $A$ is solvable.   
   		\end{subclaim}
   		
   		\begin{proof}[Proof of Subclaim \ref{cl:solution}] Starting from any equation, set values to the variables one-by-one, continuing with any equation which has both set and unset values if there is such. 
   		\end{proof}	
   		
   		Using this, we can define a map $\psi_r: p^{i_0} \to \mathbb{F}$ so that if $u_j=s_i\concatt (j) \concatt c$ then 
   		\[\sum_{0 \leq j \leq p-1} \psi_r(u_j)+ \sum_{0 \leq j \leq p-1}\phi(x_{pr(i)+j})=1\]
        for every $i\le i_0$.
   		Finally, for $u \in p^\N$ letting $\bar{\phi}(x_{(r,u)})=\psi_r(\restriction{u}{i_0})$ works as before. 
   	\end{proof}
    
   	This finishes the construction of our gadget and the proof of Proposition \ref{p:allp}
   \end{proof}
   
   Using Proposition \ref{p:allp} the same way as Proposition \ref{p:all} has been used, yields Theorem \ref{t:reductionp}. Finally, combining this with Theorem \ref{t:hypercomplexp} and Proposition \ref{p:prel} gives the main result:
   \begin{corollary}[Theorem \ref{t:main}] $CSP_{B}(\mathbb{F})$ is uniformly $\mathbf{\Sigma}^1_2$-complete.  
   \end{corollary}

   \section{Further problems and observations}
   \label{s:problems}
   	Homomorphism problems that do not reduce systems of linear equations are also an important class in the theory of CSPs: they coincide with the class of so called \emph{bounded width} CSPs, and they are solvable by relatively simple algorithms, called datalog programs (see, e.g., \cite{barto2014constraint} or \cite{brady2022notes}). It follows from Thornton's work and our main theorem that homomorphism problems which are not bounded width are necessarily $\mathbf{\Sigma}^1_2$-complete in the Borel case. 
   	
   	Let us reiterate a possible scenario suggested in Thornton's work, namely, that bounded width is exactly the boundary between easy and hard problems in the Borel case.
   	
   	\begin{conjecture}
   		A homomorphism problem is $\mathbf{\Pi}^1_1$ iff it is not $\mathbf{\Sigma}^1_2$-complete iff it is bounded width. 
   	\end{conjecture}
   
   	Since polymorphisms play a crucial role in the theory of CSPs, it is to be expected that they will play such a role in the Borel context. Recall that a \emph{polymorphism} of $H$ is just a homomorphism $H^n \to H$, where $H^n$ is the categorical power of $H$. Clearly, projections are always polymorphisms, and, roughly speaking, the CSP Dichotomy says that the homomorphism problem to $H$ is in $P$, once $H$ admits a non-trivial polymorphism, and otherwise it is NP-complete. 
   	
   	In order to Borelize this notion, it is natural to consider \emph{Borel polymorphisms} from $H^\N \to H$. A first attempt in defining a non-triviality notion for Borel polymorphisms could be as follows: call a Borel polymorphism \emph{non-trivial} if it is not continuous. We have the following observation.

   	\begin{proposition}
   		\begin{itemize}
   			\item There is no non-trivial Borel polymorphism $K_3^\N \to K_3$.
   			\item There is no non-trivial Borel polymorphism $\mathbb{F}_2^\N \to \mathbb{F}_2$.
   			\item There is a non-trivial Borel polymorphism $K_2^\N \to K_2$.
   		\end{itemize}
   	\end{proposition}
   	
   	Hence, it seems that such a notion (or a similar one) might be a good explanation for the complexity phenomenon discussed in this paper. Thus, the next problem is rather natural (see also \cite{chen2024clones} for related recent results).
   	
   	\begin{problem}
   		Develop the theory of Borel polymorphisms.
   	\end{problem}

	\bibliographystyle{abbrv}
	\bibliography{ref.bib}

\begin{thebibliography}{10}

\bibitem{barto2014constraint}
L.~Barto and M.~Kozik.
\newblock Constraint satisfaction problems solvable by local consistency
  methods.
\newblock {\em Journal of the ACM (JACM)}, 61(1):1--19, 2014.

\bibitem{brady2022notes}
Z.~Brady.
\newblock Notes on {CSP}s and polymorphisms.
\newblock {\em arXiv e-prints}, pages arXiv--2210, 2022.

\bibitem{brandt2021homomorphism}
S.~Brandt, Y.-J. Chang, J.~Grebík, C.~Grunau, V.~Rozhoň, and Z.~Vidnyánszky.
\newblock On homomorphism graphs, 2021.

\bibitem{bulatov2017dichotomy}
A.~Bulatov.
\newblock A dichotomy theorem for nonuniform {CSP}s.
\newblock In {\em 2017 IEEE 58th Annual Symposium on Foundations of Computer
  Science (FOCS)}, pages 319--330. IEEE, 2017.

\bibitem{krokhin2005complexity}
A.~Bulatov, P.~Jeavons, and A.~Krokhin.
\newblock The complexity of constraint satisfaction: an algebraic approach.
\newblock {\em Proceedings of the NATO Advanced Study Institute on Structural
  Theory of Automata, Semigroups and Universal Algebra Montreal, Quebec, Canada
  7--18 July 2003}, pages 181--213, 2005.

\bibitem{benen}
R.~Carroy, B.~D. Miller, D.~Schrittesser, and Z.~Vidny\'anszky.
\newblock Minimal definable graphs of definable chromatic number at least
  three.
\newblock {\em Forum Math. Sigma}, 9:e7, 2021.

\bibitem{chen2024clones}
R.~Chen and I.~Ziba.
\newblock Clones of {B}orel boolean functions.
\newblock {\em arXiv preprint arXiv:2407.06719}, 2024.

\bibitem{frisch2024hyper}
J.~Frisch, F.~Shinko, and Z.~Vidnyanszky.
\newblock Hyper-hyperfiniteness and complexity.
\newblock {\em arXiv preprint arXiv:2409.16445}, 2024.

\bibitem{kechrisclassical}
A.~S. Kechris.
\newblock {\em Classical descriptive set theory}, volume 156 of {\em Graduate
  Texts in Mathematics}.
\newblock Springer-Verlag, New York, 1995.

\bibitem{kechris_marks2016descriptive_comb_survey}
A.~S. Kechris and A.~S. Marks.
\newblock Descriptive graph combinatorics.
\newblock {\em
  \href{http://www.math.caltech.edu/~kechris/papers/combinatorics20book.pdf}{available
  online}}, 2020.

\bibitem{KST}
A.~S. Kechris, S.~Solecki, and S.~Todor{\v{c}}evi{\'c}.
\newblock Borel chromatic numbers.
\newblock {\em Adv. Math.}, 141(1):1--44, 1999.

\bibitem{marks2022measurable}
A.~S. Marks.
\newblock Measurable graph combinatorics.
\newblock In {\em Proc. Int. Cong. Math}, volume~3, pages 1488--1502, 2022.

\bibitem{moschovakis2009descriptive}
Y.~N. Moschovakis.
\newblock {\em Descriptive set theory}, volume 155 of {\em Mathematical Surveys
  and Monographs}.
\newblock American Mathematical Society, Providence, RI, second edition, 2009.

\bibitem{pikhurko2021descriptive_comb_survey}
O.~Pikhurko.
\newblock Borel combinatorics of locally finite graphs.
\newblock {\em Surveys in Combinatorics 2021, the 28th British Combinatorial
  Conference}, pages 267--319, 2021.

\bibitem{sabok}
M.~Sabok.
\newblock Complexity of {R}amsey null sets.
\newblock {\em Adv. Math.}, 230(3):1184--1195, 2012.

\bibitem{thornton2022algebraic}
R.~Thornton.
\newblock An algebraic approach to borel csps.
\newblock {\em arXiv preprint arXiv:2203.16712}, 2022.

\bibitem{todorvcevic2021complexity}
S.~Todor{\v{c}}evi{\'c} and Z.~Vidny{\'a}nszky.
\newblock A complexity problem for {B}orel graphs.
\newblock {\em Invent. Math.}, 226:225–249, 2021.

\bibitem{zhuk2020proof}
D.~Zhuk.
\newblock A proof of the {CSP} dichotomy conjecture.
\newblock {\em Journal of the ACM (JACM)}, 67(5):1--78, 2020.

\end{thebibliography}
	
\end{document}